\documentclass[12pt,a4paper]{amsart}
\usepackage{amsmath}
\usepackage{drpack}
\usepackage[english]{babel}
\usepackage{verbatim}
\usepackage[shortlabels]{enumitem}

\newtheoremstyle{mio}%
{}{} 
{\itshape}{} 
{\bfseries}{.}{ } 
{#1 #2\thmnote{~\mdseries(#3)}} 
\theoremstyle{mio}
\newtheorem{teor}{Theorem}[section]
\newtheorem{cor}[teor]{Corollary}
\newtheorem{prop}[teor]{Proposition}
\newtheorem{lemma}[teor]{Lemma}

\newtheoremstyle{definition2}%
{}{} 
{}{} 
{\bfseries}{.}{ } 
{#1 #2\thmnote{\mdseries~ #3}} 
\theoremstyle{definition2}

\newtheorem{hyp}[teor]{Hypothesis}

\title{The Golomb topology of polynomial rings, II}
\author{Dario Spirito}
\date{\today}
\address{Dipartimento di Scienze Matematiche, Informatiche e Fisiche, Universit\`a degli Studi di Udine, Udine, Italy}
\email{dario.spirito@uniud.it}

\newcommand{\nz}{\bullet}
\newcommand{\pow}[1]{\mathrm{pow}({#1})}
\DeclareMathOperator{\car}{char}

\begin{document}

\begin{abstract}
We study the interplay of the Golomb topology and the algebraic structure in polynomial rings $K[X]$ over a field $K$. In particular, we focus on infinite fields $K$ of positive characteristic such that the set of irreducible polynomials of $K[X]$ is dense in the Golomb space $G(K[X])$. We show that, in this case, the characteristic of $K$ is a topological invariant, and that any self-homeomorphism of $G(K[X])$ is the composition of multiplication by a unit and a ring automorphism of $K[X]$.
\end{abstract}

\maketitle

\section{Introduction}
Let $R$ be an integral domain, i.e., a commutative unitary ring without zero-divisors. The \emph{Golomb space} $G(R)$ on $R$ is the topological space having $R^\nz:=R\setminus\{0\}$ as its base space and whose topology (the \emph{Golomb topology}) is generated by the cosets $a+I$, where $a\in R$ and $I$ is an ideal such that $a^\nz$ and $I$ are coprime, i.e., $\langle a,I\rangle=R$. This construction was originally considered on the set $\insN$ of natural numbers by Brown \cite{brown-golomb} and Golomb \cite{golomb-connectedtop,golomb-aritmtop}, as part of a series of coset topologies \cite{knopf}, and subsequently extended to arbitrary rings by \cite{clark-golomb} (following ideas introduced in \cite{bmt-golomb}) with particular focus on what happens when $R$ is a Dedekind domain with infinitely many maximal ideals. In this case, $G(R)$ is a Hausdorff space that is not regular, and is a connected space that is disconnected at each of its points.

An interesting question is how much the topological structure of $G(R)$ reflects the algebraic structure of $R$: for example, it is an open question whether the fact that $G(R)$ and $G(S)$ are homeomorphic implies that $R$ and $S$ are isomorphic (as rings). Relatedly, one can ask if there are self-homeomorphisms of $G(R)$ besides the one arising from the algebraic structure, i.e., multiplication by units and automorphisims of $R$ (and their compositions).

These problems were studied in \cite{golomb-almcyc} for $R=\insZ$, showing that the only self-homeomorphisms are the trivial ones (the identity and the multiplication by $-1$) \cite[Theorem 7.7]{golomb-almcyc} and that $G(\insZ)\simeq G(R)$ cannot happen when $R\neq\insZ$ is contained in the algebraic closure of $\insQ$ \cite[Theorem 7.8]{golomb-almcyc}; a variant of its method showed that $\insN$, with the Golomb topology, is rigid, i.e., it does not have any nontrivial self-homeomorphism \cite{rigidity-N}. As a second case, \cite{golomb-poly} studied the space $G(R)$ when $R=K[X]$ is a polynomial ring over a field $K$, showing that several algebraic properties of $K$ (for example, having positive or zero characteristic, being algebraically or separably closed) imply different properties on the Golomb space. In particular, it was shown that if $K,K'$ are fields of positive characteristic that are algebraic over their base field, then $G(K[X])\simeq G(K'[X])$ imply $K\simeq K'$ when one of them is algebraically closed \cite[Theorem 5.11 and Corollary 7.2]{golomb-poly} and when they have the same characteristic \cite[Theorem 7.5]{golomb-poly}.

In this paper, we improve these results in two ways. In Section \ref{sect:homeo}, we show that, if the set of irreducible polynomials of $K[X]$ is dense in $G(K[X])$, then the characteristic of $K$ can be detected from the Golomb topology; that is, if $K,K'$ satisfy these hypothesis and $G(K[X])\simeq G(K'[X])$, then the characteristic of $K$ and $K'$ are equal (Theorem \ref{teor:charp=q}). As a consequence, we show that if $G(K[X])\simeq G(K'[X])$ and $K$ is algebraic over $\ins{F}_p$, then $K$ and $K'$ must be isomorphic (Theorem \ref{teor:algFp}). In Section \ref{sect:auto}, we concentrate on self-automorphisms of the Golomb space $G(K[X])$ and show that (under the above density hypothesis, and the condition that the characteristic of $K$ is positive) all such self-automorphisms are algebraic in nature, being compositions of a multiplication by a unit and a ring automorphism of $K[X]$ (Theorem \ref{teor:auto}). The key to both results is a variant of the proofs of \cite[Lemma 5.10]{bmt-golomb} and \cite[Theorem 6.7]{golomb-poly}, which allows to prove that, under appropriate hypothesis, any homeomorphism of Golomb spaces respects powers, i.e., $h(a^n)=h(a)^n$ for every $a\in G(R)$, $n\inN$.

\medskip

Throughout the paper, $K$ is a field, and if $R$ is an integral domain then $G(R)$ denotes the Golomb space as defined above. Given a set $X\subseteq R$, we set $X^\nz:=X\setminus\{0\}$. By a ``polynomial ring'', we always mean a polynomial ring in one variable. We denote by $U(R)$ the set of the units of $R$, so that $U(K[X])=K^\nz$.

If $h:G(R)\longrightarrow G(S)$ is a homeomorphism, then $h$ sends units to units \cite[Theorem 13]{clark-golomb} and prime ideals to prime ideals (i.e., if $P$ is a prime ideal of $R$, then $h(P^\nz)\cup\{0\}$ is a prime ideal of $S$; equivalently, $h(P^\nz)=Q^\nz$ for some prime ideal $Q$ of $S$) \cite[Theorem 3.6]{golomb-poly}.

If $x$ is a prime element of $R$, we set $\pow{x}:=\{ux^n\mid u\in U(R),n\inN\}$. If $h:G(R)\longrightarrow
G(S)$ is a homeomorphism, then $h(\pow{x})=\pow{h(x)}$ (see \cite[Section 2]{golomb-poly}).

If $P$ is a prime ideal of $R$, the \emph{$P$-topology} on $R\setminus P$ is just the $P$-adic topology, i.e., the topology having the sets $a+P^n$, for $n\inN$, as a local basis for $a$. The $P$-topology can be recovered from the Golomb topology, and thus, for any prime ideal $P$ of $R$, a homeomorphism $h:G(R)\longrightarrow G(S)$ restricts to a homeomorphism between $R\setminus P$ endowed with the $P$-topology and $S\setminus Q$ endowed with the $Q$-topology (where $h(P^\nz)=Q^\nz$) \cite[Theorem 4.3]{golomb-almcyc}.

We say that $R$ is \emph{Dirichlet} if the set of irreducible elements of $R$ is dense, with respect to the Golomb topology. If $R=K[X]$ is Dirichlet, where $K$ is a field, we say for brevity that $K$ itself is Dirichlet. This happens, for example, when $K$ is pseudo-algebraically closed and admits separable irreducible polynomials of arbitrary large degree \cite[Theorem A]{dirichlet-polynomial}; in particular, it happens when $K$ is an algebraic extension of a finite field that is not algebraically closed \cite[Corollary 11.2.4]{field_arithmetic}.

\section{Detecting the characteristic}\label{sect:homeo}
In this section, we show how to detect the characteristic of a polynomial ring from its Golomb topology.

Given an element $t\in R^\nz$, we denote by $\sigma_t$ the multiplication by $t$, i.e., 
\begin{equation*}
\begin{aligned}
\sigma_t\colon G(R) & \longrightarrow G(R),\\
y& \longmapsto ty.
\end{aligned}
\end{equation*}
It is easy to see that the map $\sigma_t$ is always continuous, and that it is a homeomoprhism if $t$ is a unit of $R$.

If $R$ is a polynomial ring over a Dirichlet field, and $h:G(R)\longrightarrow G(R)$ is a self-homeomorphism such that $h(P^\nz)=P^\nz$ for every prime ideal $P$, then $h$ must be equal to $\sigma_u$ for some unit $u$ \cite[Proposition 7.4]{golomb-poly}; this result can be improved in the following way.

\begin{prop}\label{prop:multiplic-unit}
Let $R,S$ be polynomials rings over Dirichlet fields. If $h:G(R)\longrightarrow G(S)$ is a homeomorphism such that $h(1)=1$, then $h(uz)=h(u)h(z)$ for every $u\in U(R)$ and $z\in R^\nz$.
\end{prop}
\begin{proof}
Let $u$ be a unit of $R$. Consider the map $h':=h^{-1}\circ\sigma_{h(u)}\circ h$, that is,
\begin{equation*}
h'(z)=h^{-1}(h(u)h(z))
\end{equation*}
for every $z\in R^\nz$. Since $u$ is a unit, so is $h(u)$, and thus $\sigma_{h(u)}$ is a self-homeomorphism of $G(S)$; therefore, $h'$ is a homeomorphism of $G(R)$.

Let now $P$ be a prime ideal of $R$; then, $h(P^\nz)=Q^\nz$ for some prime ideal $Q$ of $S$. For every $z\in P$, we have $h(u)h(z)\in Q$, and thus $h'(z)=h^{-1}(h(u)h(z))\in P$. Hence, $h'(P)\subseteq P$, i.e., $h$ fixes every prime ideal of $R$; by \cite[Proposition 7.4]{golomb-poly}, $h'=\sigma_t$ for some unit $t$ of $R$. However, since $h(1)=1$,
\begin{equation*}
h'(1)=h^{-1}(h(u)h(1))=h^{-1}(h(u))=u,
\end{equation*}
and thus it must be $t=u$, i.e., $h'(z)=uz$ for every $z\in R^\nz$; it follows that, $h(uz)=h(u)h(z)$ for every $z\in R^\nz$.
\end{proof}

We now want to study the relationship between a homeomorphism of Golomb topologies and powers of elements. For an element $z\in R^\nz$, we denote by $z^\insN$ the set of powers of $z$, i.e., $z^\insN:=\{z^n\mid n\inN\}$.
\begin{lemma}\label{lemma:Xhdenso}
Let $R,S$ be polynomial rings over Dirichlet fields, and let $h:G(R)\longrightarrow G(S)$ be a homeomorphism such that $h(1)=1$. The set
\begin{equation*}
\mathcal{X}_h:=\{z\in R\mid z\text{~is irreducible and~}h(z^\insN)=h(z)^\insN\}
\end{equation*}
is dense in $G(R)$.
\end{lemma}
\begin{proof}
Let $Q$ be a prime ideal of $S$, and let $P$ be the prime ideal of $R$ such that $h(P^\nz)=Q^\nz$. Then, $h$ restricts to a homeomorphism between the $P$-topology on $R\setminus P$ and the $Q$-topology on $S\setminus Q$; it follows that for every $m\inN$ there is an $n$ such that $h(1+P^n)\subseteq 1+Q^m$. Let $z$ be any irreducible element in $1+P^n$ (which exists since $R$ is Dirichlet); then, for every $t\inN$, we have $z^t\in 1+P^n$ too, and thus also $h(z^t)\in 1+Q^m$. Moreover, $z^t\in\pow{z}$, and thus there are a unit $u$ of $S$ and an integer $\alpha(t)$ such that $h(z^t)=uh(z)^{\alpha(t)}$. Since $U(S)\longrightarrow S/Q$ is injective, it follows that $u=1$ for every $t$. Hence, $h(z^\insN)=h(z)^\insN$.

Let now $\Omega=c+I$ be a subbasic open set of $G(R)$, and let $d+J$ be a subbasic open set contained in $h(\Omega)$. Let $Q$ be a prime ideal of $S$ that is coprime with $J$; then, the prime $P$ such that $h(P^\nz)=Q^\nz$ is coprime with $I$. Since $R$ is Dirichlet, $(c+I)\cap(1+P^n)$ contains irreducible elements for every $n$; in particular, is $n$ is large enough, every irreducible polynomial in $1+P^n$ is in $\mathcal{X}_h$, and thus $\mathcal{X}_h\cap(c+I)$ is nonempty. Thus $\mathcal{X}_h$ is dense. 
\end{proof}

The following proof follows the one of \cite[Theorem 6.8]{golomb-poly}, which in turn was based on the one of \cite[Lemma 5.10]{bmt-golomb}. 
\begin{lemma}\label{lemma:monogenic}
Let $R,S$ be polynomial rings over Dirichlet fields, and let $h:G(R)\longrightarrow G(S)$ be a homeomorphism such that $h(1)=1$. For every $a\in R$, we have $h(a^\insN)=h(a)^\insN$.
\end{lemma}
\begin{proof}
If $a$ is a unit, the statement follows from Proposition \ref{prop:multiplic-unit}. Take any $a\in R$, $a\notin U(R)$, and let $b:=h(a)$. We first claim that $h(a^\insN)\subseteq b^\insN$. Take any $n\inN$ and let $f:G(R)\longrightarrow G(R)$ be defined as $f(y)=y^n$, and let $\phi:=h\circ f\circ h^{-1}$. Then, $\phi:G(S)\longrightarrow G(S)$ is continuous (since $f$ is continuous) and $\phi(P^\nz)\subseteq P^\nz$ for every prime ideal $P$ of $R$. Let
\begin{equation*}
c:=\phi(b)=h(f(a))=h(a^n),
\end{equation*}
and suppose that $c\notin b^\insN$. Take an integer $k$ such that $\deg b^k>\deg h(a^n)=\deg c$; then, since $b$ and $b^k-1$ are coprime, so are $c$ and $b^k-1$ (since $c$ and $b$ belong to the same prime ideals, by \cite[Theorem 3.6]{golomb-almcyc}), and thus $\Omega:=c+(b^k-1)S$ is an open set. Since $\phi$ is open, there is a neighborhood $b+dS$ of $b$ such that $\phi(b+dS)\subseteq\Omega$.

By Lemma \ref{lemma:Xhdenso}, we can find an irreducible polynomial $z\in h^{-1}(b+d(b^k-1)S)$ such that $h(z^\insN)=h(z)^\insN$; setting $y:=h(z)$, we have
\begin{equation*}
\phi(y)=(h\circ f\circ h^{-1})(h(z))=h(z^n)=h(z)^l
\end{equation*}
for some integer $l$. Thus, we have
\begin{equation*}
\phi(y)\in \phi(b+dS)\subseteq\Omega=c+(b^k-1)S
\end{equation*}
and
\begin{equation*}
\phi(y)=h(z)^l\in (b+d(b^k-1)S)^l\subseteq b^l+(b^k-1)S;
\end{equation*}
hence, $c\equiv b^l\mod(b^k-1)S$. If $l=ik+j$, we have $b^j\equiv b^l\bmod(b^k-1)S$; hence, we have $c\equiv b^j\mod(b^k-1)S$ for some $0\leq j<k$. However, by hypothesis, $c\neq b^j$, and the degree of both $c$ and $b^j$ is less than the degree of $b^k-1$; this is a contradiction, and thus we must have $h(a^\insN)\subseteq b^\insN$.

The opposite inclusion is obtained using the homeomorphism $h^{-1}$. Thus, $h(a^\insN)=b^\insN=h(a)^\insN$, as claimed.
\end{proof}

\begin{lemma}\label{lemma:exponent}
Let $R$ be an integral domain. If $U(R)$ is infinite, then its exponent is infinite.
\end{lemma}
\begin{proof}
Suppose that the exponent of $U(R)$ is finite, say equal to $n$. Since $R$ is a domain, there can be at most $n$ units of order divisible by $n$ (the roots of $X^n-1=0$); however, this is impossible since $U(R)$ is infinite. Hence, the exponent is infinite.
\end{proof}

\begin{prop}\label{prop:han}
Let $R,S$ be polynomial rings over infinite Dirichlet fields, and let $h:G(R)\longrightarrow G(S)$ be a homeomorphism such that $h(1)=1$. For every $a\in R$ and every $n\inN$, we have $h(a^n)=h(a)^n$.
\end{prop}
\begin{proof}
If $a$ is a unit the claim follows from Proposition \ref{prop:multiplic-unit}. Take $a\in R\setminus U(R)$ and $n\inN$; by Lemma \ref{lemma:monogenic}, we have $h(a^n)=h(a)^s$ for some $s\inN$. By hypothesis, $R$ has infinitely many units, and thus by Lemma \ref{lemma:exponent} there is a unit $u$ whose order is larger than $n$ and $s$; then, using Proposition \ref{prop:multiplic-unit} we have
\begin{equation*}
h((ua)^n)=(h(ua))^t=(h(u)h(a))^t=h(u)^th(a)^t
\end{equation*}
for some $t$ and
\begin{equation*}
h((ua)^n)=h(u^na^n)=h(u^n)h(a^n)=h(u)^nh(a)^s.
\end{equation*}
The equality 
\begin{equation*}
h(u)^th(a)^t=h(u)^nh(a)^s
\end{equation*}
can only hold if $n=t=s$; in particular, $n=s$ and $h(a^n)=h(a)^n$, as claimed.
\end{proof}

\begin{teor}\label{teor:charp=q}
Let $K,K'$ be infinite Dirichlet fields. If $G(K[X])$ and $G(K'[X])$ are homeomorphic, then they have the same characteristic.
\end{teor}
\begin{proof}
If one of $K$ and $K'$ has characteristic $0$, the claim follows from \cite[Corollary 4.2]{golomb-poly}. Suppose thus that $\car K=p$ and $\car K'=q$, with $p,q>0$.

Let $h:G(K[X])\longrightarrow G(K'[X])$ be a homeomorphism such that $h(1)=1$, and let $a$ be an irreducible element. Let $f$ be a factor of $a-1$. By \cite[Proposition 5.5]{golomb-poly}, the sequence $a^{p^n}$ converges to $1$ in the $(f)$-topology, and thus also $h(a^{p^n})$ converges to $1$ in the $h((f))$-topology. By Proposition \ref{prop:han}, $h(a^{p^n})=h(a)^{p^n}$. By \cite[Proposition 5.5]{golomb-poly}, it follows that $v_q(p^n)\longrightarrow\infty$, where $v_q$ is the $q$-adic valuation; thus, it must be $q=p$. The claim is proved.
\end{proof}

We note that the proof of the previous theorem actually needs only \emph{one} irreducible element $a$ such that $h(a^n)=h(a)^n$ for every $n$.

\begin{teor}\label{teor:algFp}
Let $K,K'$ be fields. If $K$ is algebraic over $\ins{F}_p$ and $G(K[X])\simeq G(K'[X])$, then $K\simeq K'$.
\end{teor}
\begin{proof}
If $K$ is finite then $|U(K[X])|=|K|-1$. Since the set of units is preserved by homeomorphisms of the Golomb topology, $K$ and $K'$ must have the same cardinality, and thus they are isomorphic. Suppose $K,K'$ are infinite.

By \cite[Corollary 4.2]{golomb-poly}, $K'$ has positive characteristic, and by \cite[Corollary 7.2]{golomb-poly} $K'$ must be algebraic over its base field $\ins{F}_q$.

If $K$ is algebraically closed, then the set $G_1(K[X]):=\{z\in K[X]\mid z$ is contained in a unique prime ideal$\}$ is not dense in $G(K[X])$ \cite[Proposition 5.2(a)]{golomb-poly}; conversely, if $K$ is not algebraically closed then $G_1(K[X])$ is dense since it contains the irreducible polynomials. The same holds for $K'$. Since a homeomorphism of Golomb topologies sends $G_1(K[X])$ to $G_1(K'[X])$ \cite[Theorem 3.6(b)]{golomb-almcyc}, it follows that if $K$ is algebraically closed then so is $K'$, and in this case $p=q$ by \cite[Theorem 5.11]{golomb-poly}; hence $K\simeq K'$.

If $K$ and $K'$ are not algebraically closed, they are pseudo-algebraically closed fields containing separable irreducible polynomials of arbitrarily large degree \cite[Theorem A]{dirichlet-polynomial}; by Theorem \ref{teor:charp=q}, it follows that $p=q$. By \cite[Theorem 7.5]{golomb-poly}, it follows that $K\simeq K'$.
\end{proof}

\section{Self-homeomorphisms and automorphisms}\label{sect:auto}
In this section, we concentrate on the study of self-homeomorphisms of the Golomb space $G(K[X])$. We shall work under the following assumptions:
\begin{hyp}\label{hyp-auto}
~
\begin{itemize}
\item $K$ is an infinite field of characteristic $p>0$;
\item $R:=K[X]$;
\item $R$ is Dirichlet, i.e., the set of irreducible polynomials is dense in $G(R)$;
\item $h:G(R)\longrightarrow G(R)$ is a self-homeomorphism such that $h(1)=1$.
\end{itemize}
\end{hyp}

In particular, under these hypothesis, we have $h(a^n)=h(a)^n$ for every $a\in R$ and $n\inN$, by Proposition \ref{prop:han}.

Given a set $X\subseteq R$, we define the \emph{$p$-radical} of $X$ as
\begin{equation*}
\rad_p(X):=\{f\in R\mid f^{p^n}\in X\text{~for every large~}n\}.
\end{equation*}
This construction is invariant under $h$, in the following sense.
\begin{lemma}\label{lemma:radp-inv}
Assume Hypothesis \ref{hyp-auto}. For every $X\subseteq R^\nz$, we have $h(\rad_p(X))=\rad_p(h(X))$.
\end{lemma}
\begin{proof}
If $f\in\rad_p(X)$, then $f^{p^n}\in X$ for every $n\geq N$, and thus $h(f^{p^n})\in h(X)$. By Proposition \ref{prop:han}, it follows that $h(f)^{p^n}\in h(X)$ for $n\geq N$, and thus $h(f)\in \rad_p(h(X))$. Conversely, if $g=h(f)\in\rad_p(h(X))$, then $g^n=h(f)^{p^n}=h(f^{p^n})$ belongs to $h(X)$ for large $n$, and thus applying $h^{-1}$ we have $f^{p^n}\in X$ for large $n$. Thus $f\in\rad_p(X)$ and $g\in h(\rad_p(X))$.
\end{proof}

\begin{lemma}\label{lemma:radp-1Pm}
Let $P$ a prime ideal of a ring $R$ of characteristic $p$. For every $m\inN^+$, we have $\rad_p(1+P^m)=1+P$.
\end{lemma}
\begin{proof}
If $t\in 1+P$, then $t=1+x$ with $x\in P$; if $p^n\geq m$, then $t^{p^n}=(1+x)^{p^n}=1+x^{p^n}\in 1+P^m$, and thus $t\in\rad_p(1+P^m)$. Conversely, if $t\in\rad_p(1+P^m)$ then $t^{p^n}\in 1+P^m$ for some $m$; in particular, $t^{p^n}\in 1+P$, and thus $t^{p^n}-1=(t-1)^{p^n}\in P$. Hence $t-1\in P$, i.e., $t\in 1+P$.
\end{proof}

\begin{prop}\label{prop:h1P}
Assume Hypothesis \ref{hyp-auto}, and let $P,Q$ be prime ideals of $R$ with $h(P^\nz)=Q^\nz$. Then, $h(1+P)=1+Q$.
\end{prop}
\begin{proof}
Since $h(1)=1$, $h(1+P)$ is open in the $Q$-topology, and thus $1+Q^m\subseteq h(1+P)$ for some $m$. Therefore,
\begin{equation*}
1+Q=\rad_p(1+Q^m)\subseteq\rad_p(h(1+P))=h(\rad_p(1+P))=h(1+P)
\end{equation*}
using Lemmas \ref{lemma:radp-inv} and \ref{lemma:radp-1Pm}. Applying the same reasoning to $h^{-1}$ gives $1+P\subseteq h^{-1}(1+Q)$, i.e., $h(1+P)\subseteq 1+Q$. Hence, $h(1+P)=1+Q$, as claimed.
\end{proof}

\begin{cor}\label{cor:huP}
Assume Hypothesis \ref{hyp-auto},and let $P,Q$ be prime ideals of $R$ with $h(P^\nz)=Q^\nz$. If $u\in K^\nz$, then $h(u+P)=h(u)+Q$. In particular, if $f\equiv u\bmod P$ for some unit $u$, then $h(f)\equiv h(u)\bmod Q$.
\end{cor}
\begin{proof}
By Propositions \ref{prop:multiplic-unit}, we have $u+P=u(1+P)$; by Proposition \ref{prop:h1P}, it follows that
\begin{equation*}
h(u+P)=h(u)h(1+P)=h(u)(1+Q)=h(u)+Q,
\end{equation*}
as claimed. The ``in particular'' statement follows from the fact that $f\equiv u\bmod P$ is equivalent to $f\in u+P$.
\end{proof}

\begin{cor}\label{cor:linear}
Assume Hypothesis \ref{hyp-auto}. If $f$ is a linear polynomial, so is $h(f)$.
\end{cor}
\begin{proof}
Since $f$ is linear, $(f)$ is prime and every polynomial $g\notin(f)$ is equivalent to a unit modulo $(f)$. By Corollary \ref{cor:huP}, it follows that every $g\notin(h(f))$ is equivalent to a unit modulo $(h(f))$; hence, $h(f)$ must be linear too.
\end{proof}

In particular, by Corollary \ref{cor:linear}, $h(X)$ is a linear polynomial; it follows that there is an automorphism $\sigma$ of $R$ sending $h(X)$ to $h$. Since $\sigma$ restricts to a self-homeomorphism of $G(R)$, passing from $h$ to $H:=\sigma\circ h$ we obtain a self-homeomorphism of $G(R)$ that fixes both $1$ and $X$. Thus, it is not restrictive to assume also that $h(X)=X$, as we do in the remaining part of the section.

\begin{lemma}\label{lemma:hXu}
Assume Hypothesis \ref{hyp-auto} and suppose $h(X)=X$. For every $u\in K^\nz$, we have $h(X+u)=X+h(u)$.
\end{lemma}
\begin{proof}
We have
\begin{equation*}
\begin{cases}
X+u\equiv u\bmod(X)\\
X\equiv -u\bmod(X+u).
\end{cases}
\end{equation*}
Applying $h$ to both equivalences, using $h(X)=X$ and Corollary \ref{cor:huP}, we have
\begin{equation*}
\begin{cases}
h(X+u)\equiv h(u)\bmod(X)\\
X\equiv -h(u)\bmod(h(X+u)).
\end{cases}
\end{equation*}
The first equation implies that $X$ divides $h(X+u)-h(u)$; since $h(X+u)$ is linear, it follows that $h(X+u)=vX+h(u)$ for some $v\in K^\nz$. The second equation implies that $vX+h(u)$ divides $X+h(u)$; the only possibility is $v=1$, i.e., $h(X+u)=X+h(u)$. The claim is proved.
\end{proof}

\begin{lemma}\label{lemma:autoK}
Assume Hypothesis \ref{hyp-auto} and suppose $h(X)=X$. Then, the map
\begin{equation*}
\begin{aligned}
H\colon K & \longrightarrow K,\\
a& \longmapsto \begin{cases}h(a) & \text{if~}a\neq 0\\
0 & \text{if~}a=0\end{cases}
\end{aligned}
\end{equation*}
is an automorphism of $K$.
\end{lemma}
\begin{proof}
Let $a,b\in K$. If $a=0$ or $b=0$ then clearly $H(ab)=H(a)H(b)$ and $H(a+b)=H(a)+H(b)$.

Suppose $a\neq 0\neq b$. Then, $H=h$ on these values, and $h(ab)=h(a)h(b)$ by Proposition \ref{prop:multiplic-unit}. Furthermore,
\begin{equation*}
X+a+b\equiv b\bmod(X+a);
\end{equation*}
applying $h$ and using Lemma \ref{lemma:hXu}, we have
\begin{equation*}
X+h(a+b)\equiv h(b)\bmod(X+h(a)),
\end{equation*}
that is, $X+h(a)$ divides $X+h(a+b)-h(b)$. Thus, it must be $X+h(a)=X+h(a+b)-h(b)$ and $h(a)=h(a+b)-h(b)$, that is, $h(a)+h(b)=h(a+b)$. It follows that $H$ is an automorphism of $K$, as claimed.
\end{proof}

\begin{teor}\label{teor:auto}
Let $K$ is an infinite Dirichlet field of positive characteristic, and let $h$ be a self-homeomorphism of $G(K[X])$. Then, there are a unit $u\in K^\nz$ and an automorphism $\sigma$ of $K[X]$ such that $h(f)=u\sigma(f)$ for every $f\in K[X]^\nz$.
\end{teor}
\begin{proof}
By \cite[Theorem 13]{clark-golomb}, $h(1)=u$ is a unit of $R$; since the multiplication $\sigma_u$ is a self-homeomorphism of $G(R)$, the map $h_1:=\sigma_u^{-1}\circ h$ is a self-homeomorphism such that $h_1(1)=1$.

By Corollary \ref{cor:linear}, $h_1(X)$ is a linear polynomial, and thus there is an automorphism $\sigma_1$ of $R$ such that $h_1(X)=\sigma_1(X)$. Thus, $h_2:=\sigma_1^{-1}\circ h_1$ is again a self-homeomorphism of $G(R)$, and furthermore both $1$ and $X$ are fixed points of $h_2$.

By Lemma \ref{lemma:autoK}, the restriction of $h_2$ to $K$ is an automorphism $H$; this map can be extended to an automorphism $\sigma_2$ of $R$ by setting $\sigma_2\left(\sum_ia_iX^i\right)=\sum_iH(a_i)X^i$. Thus, $h_3:=\sigma_2^{-1}\circ h_2$ is a self-homeomorphism of $G(R)$ such that $h_3(u)=u$ for every $u\in K^\nz$ and $h_3(X)=X$. By Lemma \ref{lemma:hXu}, it follows that $h_3(X+u)=X+u$ for every $u\in K^\nz$.

Let now $f\in R$, and take any $t\in K$. Then,
\begin{equation*}
\begin{cases}
f\equiv f(t)\bmod(X-t)\\
h_3(f)\equiv h_3(f)(t)\bmod(X-t).
\end{cases}
\end{equation*}
Since $h_3((X-t))=(X-t)$, by Corollary \ref{cor:huP} the first equivalence also implies $h_3(f)\equiv h_3(f(t))\bmod(X-t)$. Hence,  $h_3(f)(t)=h_3(f(t))=f(t)$. Since this happens for every $t\in K$, and $K$ is infinite, it follows that $f=h_3(f)$, that is, $h_3$ is the identity on $G(R)$.

Going back to the definition,
\begin{equation*}
h_3=\sigma_2^{-1}\circ h_2=\sigma_2^{-1}\circ \sigma_1^{-1}\circ h_1=\sigma_2^{-1}\circ \sigma_1^{-1}\circ\sigma_u^{-1}\circ h,
\end{equation*}
i.e., $h=\sigma_u\circ\sigma_1\circ\sigma_2$. Since $\sigma_u(x)=ux$ for every $x$, setting $\sigma:=\sigma_1\circ\sigma_2$ (which is still an automorphism of $R$), we obtain $h(f)=u\sigma(f)$ for every $f\in G(R)$, as claimed.
\end{proof}

\bibliographystyle{plain}
\bibliography{/bib/articoli,/bib/libri,/bib/miei}

\begin{thebibliography}{10}

\bibitem{bmt-golomb}
Taras Banakh, Jerzy Mioduszewski, and S{\l}awomir Turek.
\newblock On continuous self-maps and homeomorphisms of the {G}olomb space.
\newblock {\em Comment. Math. Univ. Carolin.}, 59(4):423--442, 2018.

\bibitem{rigidity-N}
Taras Banakh, Dario Spirito, and S\l~awomir Turek.
\newblock The {G}olomb space is topologically rigid.
\newblock {\em Comment. Math. Univ. Carolin.}, 62(3):347--360, 2021.

\bibitem{dirichlet-polynomial}
Lior Bary-Soroker.
\newblock Dirichlet's theorem for polynomial rings.
\newblock {\em Proc. Amer. Math. Soc.}, 137(1):73--83, 2009.

\bibitem{brown-golomb}
Morton Brown.
\newblock A countable connected {H}ausdorff space.
\newblock In L.W. Cohen, editor, {\em The {A}pril meeting in {N}ew {Y}ork},
  volume~4, pages 330--371. Bull. Amer. Math. Soc., 1953.
\newblock Abstract 423.

\bibitem{clark-golomb}
Pete~L. Clark, Noah Lebowitz-Lockard, and Paul Pollack.
\newblock A note on {G}olomb topologies.
\newblock {\em Quaest. Math.}, 42(1):73--86, 2019.

\bibitem{field_arithmetic}
Michael~D. Fried and Moshe Jarden.
\newblock {\em Field arithmetic}, volume~11 of {\em Ergebnisse der Mathematik
  und ihrer Grenzgebiete. 3. Folge. A Series of Modern Surveys in Mathematics
  [Results in Mathematics and Related Areas. 3rd Series. A Series of Modern
  Surveys in Mathematics]}.
\newblock Springer-Verlag, Berlin, second edition, 2005.

\bibitem{golomb-connectedtop}
Solomon~W. Golomb.
\newblock A connected topology for the integers.
\newblock {\em Amer. Math. Monthly}, 66:663--665, 1959.

\bibitem{golomb-aritmtop}
Solomon~W. Golomb.
\newblock Arithmetica topologica.
\newblock In {\em General {T}opology and its {R}elations to {M}odern {A}nalysis
  and {A}lgebra ({P}roc. {S}ympos., {P}rague, 1961)}, pages 179--186. Academic
  Press, New York; Publ. House Czech. Acad. Sci., Prague, 1962.

\bibitem{knopf}
John Knopfmacher and Stefan Porubsky.
\newblock Topologies related to arithmetical properties of integral domains.
\newblock {\em Exposition. Math.}, 15(2):131--148, 1997.

\bibitem{golomb-almcyc}
Dario Spirito.
\newblock The {G}olomb topology on a {D}edekind domain and the group of units
  of its quotients.
\newblock {\em Topology Appl.}, 273:107101, 20, 2020.

\bibitem{golomb-poly}
Dario Spirito.
\newblock The {G}olomb topology of polynomial rings.
\newblock {\em Quaest. Math.}, 44(4):447--468, 2021.

\end{thebibliography}
\end{document}